\documentclass[12pt,a4paper,reqno]{amsart}
\usepackage[utf8]{inputenc}
\usepackage[T1]{fontenc}
\usepackage[english]{babel}
\usepackage{amsfonts,amsmath, amssymb,latexsym,times,mathrsfs}

\usepackage{graphics}

\usepackage{graphicx}

\usepackage{ae,aecompl}
\usepackage{pstricks}
\usepackage{pstricks-add}
\usepackage{float}

\usepackage[colorlinks=true]{hyperref}
\hypersetup{ pdfstartview=FitH }

\numberwithin{equation}{section}
\def\ZZ{{\mathbb{Z}}}     \def\CC{{\mathbb{C}}}
\def\RR{{\mathbb{R}}}   \def\EE{{\mathbb{E}}} 
\def\cO{{\mathcal{O}}}

\def\scrT{{\mathscr{T}}}
\def\scrK{{\mathscr{K}}}

\def\cM{{\mathscr{M}}}
\def\scrM{{\mathscr{M}}}
\def\cC{{\mathcal{C}}}
\def\scrC{{\mathscr{C}}}
\def\scrU{{\mathscr{U}}}
\def\scrV{{\mathscr{V}}}

\def\vert{\mathbf{v}}
\def\zert{\mathbf{z}}
\def\face{\mathbf{f}}

\def\fC{{\mathfrak{C}}}

\renewcommand{\epsilon}{\varepsilon}

\newcommand{\id}{\mathrm{id}}

\newcommand{\CP}{\mathbb{CP}}

\newcommand{\cQ}{\mathcal{Q}}

\newcommand{\del}{\partial}

\renewcommand{\Im}{\mathrm{Im}}

\newcommand{\resp}{\emph{resp.}}

\newtheorem{theointro}{Theorem}

\newtheorem{corintro}[theointro]{Corollary}

\newtheorem{lemma}[subsubsection]{Lemma}
\newtheorem{prop}[subsubsection]{Proposition}
\newtheorem{cor}[subsubsection]{Corollary}
\newtheorem{theo}[subsubsection]{Theorem}

\newtheorem{dfn}[subsubsection]{Definition}

\theoremstyle{definition}

\theoremstyle{remark}
\newtheorem{rmk}[subsubsection]{Remark}

\newtheorem*{rmk*}{Remark}
\newtheorem*{rmks*}{Remarks}
\newtheorem{rmks}[subsubsection]{Remarks}


\title{Polyhedral approximation by Lagrangian and isotropic tori}

\author{Yann Rollin}
\address{Yann Rollin, Laboratoire Jean Leray, Universit\'e de Nantes}
\email{yann.rollin@univ-nantes.fr}

\begin{document}
\begin{abstract}
	We prove that every smoothly immersed $2$-torus of $\RR^{4}$ can
	be approximated, in the $C^0$-sense, by immersed polyhedral
	Lagrangian 
	tori. In the case of a smoothly immersed (resp. embedded) Lagrangian
	torus of $\RR^4$, the surface can be approximated in the $C^1$-sense by immersed
	(resp. embedded) polyhedral Lagrangian tori. Similar statements are
	proved for isotropic $2$-tori of $\RR^{2n}$.
\end{abstract}

{\Huge \sc \bf\maketitle}

\section{Introduction}

\subsection{Statement of results}
The main result of this paper is the following approximation theorem,
in the context of
piecewise linear symplectic geometry.
\begin{theointro}
	\label{theo:main}
Let $\Sigma$ be a surface diffeomorphic to the $2$-dimensional real torus
	and
	$\ell:\Sigma\to\RR^{2n}$, a smooth isotropic immersion (resp.
	embedding), for some integer $n\geq 2$. Then,
	 there exist  piecewise linear isotropic  topological immersions (resp.
	embeddings) 
	$f:\Sigma\to\RR^{2n}$, arbitrarily close to $\ell$ in the
	$C^1$-sense.
\end{theointro}
	A weak version of Theorem~\ref{theo:main} appeared first in~\cite{JRT}. 
	This improved result now deals with the stronger $C^1$-estimate and the case
	of critical dimension $n=2$, which  is perhaps the most
	interesting. Indeed this new result can be applied to manifolds,
	rather than maps, and leads to construction of large
	families  of  immersed and embedded
	polyhedral Lagrangian tori of $\RR^4$, obtained by approximation of the smooth ones,
	according to Corollaries~\ref{cor:main} and \ref{cor:hprin}.
\begin{corintro}
	\label{cor:main}
	For every integer $n\geq 2$,
	any smoothly immersed (\resp. embedded) isotropic $2$-dimensional torus of
	$\RR^{2n}$ can be approximated,  in the $C^1$-sense, by arbitrarily close
	immersed (\resp. embedded) polyhedral isotropic tori.
\end{corintro}
Combined with the h-principle, we obtain the following approximation
theorem for smooth tori: 
\begin{corintro}
	\label{cor:hprin}
	For $n\geq 2$, every smoothly immersed 
	$2$-dimensional 
	torus of
	$\RR^{2n}$  can be approximated, in $C^0$-sense, by
	arbitrarily close immersed polyhedral isotropic tori.

	If $n\geq 3$, we can replace "immersed" with "embedded" in the above
	statement.
\end{corintro}
%

\subsection{Some motivations on piecewise linear symplectic geometry}
 \emph{Flexibility} issues have attracted
 a lot of attention in 
symplectic geometry. Similarly to the \emph{sphere eversion theorem}, many flexibility
questions arise in symplectic geometry and do not have  obvious answers.
Gromov
invented powerful techniques, promoted under the name of
\emph{h-principle}, to study certain classes of underdetermined partial differential
relations~\cite{G86}.  
For example, by a  theorem of Gromov, any smoothly immersed surface of~$\RR^4$
can be approximated by Lagrangian immersed surfaces, in
the  $C^0$-sense. 
Like many flexibility results in the framework of symplectic geometry, the
proof of this fact relies  on  stability properties of symplectic
structures and ordinary differential equation techniques. 
Those properties, specific to smooth
manifolds,   are not directly available for
\emph{piecewise linear} 
or \emph{piecewise smooth manifolds}. 
Since these essential tools are missing, very little is known about
natural versions of piecewise linear  
symplectic geometry. We mention now several elementary problems of symplectic
geometry, whose piecewise linear counterparts are poorly understood:
\begin{enumerate}
			\item In spite of the fact that triangulations are
quite common in geometry and topology, it is not known whether a smooth symplectic
		manifold admits a \emph{symplectic triangulation} (cf.
		the original work of Distexhe in~\cite{Dis19} for more details). Even the case of $\CP^n$ for
$n\geq 2$ is not known. 
		 Conversely, it is not known if there are suitable
			smoothing techniques for
		piecewise linear symplectic manifolds. 
		\item By Darboux theorem, every smooth symplectic structure
			is locally standard, which shows that smooth
			symplectic manifold do not have any local
			invariants. However,  there is no 
			analogue of the Darboux
			theorem for piecewise linear symplectic manifolds.
In fact, there might even be local
obstructions for the local triviality of piecewise linear symplectic structures!
\item There are many constructions of  Lagrangian
		manifolds for smooth symplectic manifolds and their
		deformation theory is very simple. This is natural
		since Lagrangian manifolds are 
central object for the symplectic topologists
		preoccupations.  
		Yet, 
		examples of polyhedral Lagrangian manifolds are scarce.
		As the main contribution of this paper, we constructs large families of polyhedral
		Lagrangian tori of $\RR^4$ by approximation   of the smooth ones in
		Corollaries~\ref{cor:main} and \ref{cor:hprin}. 
		\item Suitable piecewise linear 
			analogues of the group of symplectic and
			Hamiltonian diffeomorphisms are not well
			understood. A first approach was made by Gratza, 
			who introduced some clever (and quite difficult)
			approximation techniques in~\cite{Gra98}. 
\end{enumerate}
 Because piecewise linear or polyhedral geometry is such natural
framework, it has attracted a lot of attention in the context of topology,
in particular as a tool 
for studying the relation between smooth and topological
manifolds.
It seems natural to consider problems of a symplectic nature in the
piecewise linear context.
For instance, numerical experiments for  
various geometric flows involving symplectic
manifolds rely on the use of
polyhedral geometry for discretization purposes. Thus,
natural questions about symplectic geometry 
arise from a numerical perspective as well.
The goal of this work, initiated in~\cite{JRT}, is an attempt to 
approximate smooth isotropic or Lagrangian
manifolds using piecewise linear geometry.
Although this is a first step, it is not clear at this stage  what is the
relationship between the smooth symplectic and piecewise linear symplectic 
worlds. Before we get a full picture tremendous efforts have yet to be made. 

The proof of Theorem~\ref{theo:main} relies
on some techniques introduced in~\cite{JRT}. 
The main idea is that the problem of finding isotropic surfaces is 
 equivalent to  the problem of finding zeroes of a certain
\emph{Donaldson moment map}, by the \emph{fixed point principle}. This may seem unusual, but this
point of view also provides effective techniques for producing solutions
via numerical methods and flows~\cite{JRT,JR}. However, a more
straightforward 
approach, which is expected to work as well,  is currently being taken care of
by Etourneau in~\cite{Et}.

\section{Background material}
In  this section, we provide some background material and
introduce some vocabulary, that have  been used in
Theorem~\ref{theo:main} and Corollaries~\ref{cor:main} and~\ref{cor:hprin},
without a proper introduction. 
Then, an overview of the techniques introduced in~\cite{JRT} is given,
since these tools play a central role in  the
proof of Theorem~\ref{theo:main}.

\subsection{Basic definitions of symplectic geometry}
The even dimensional Euclidean space $\RR^{2n}$,
identified to $ \CC^n$, comes equipped with canonical coordinates $z_j=x_j+iy_j$, for
$1\leq j\leq n$. This Euclidean
space is endowed with the \emph{canonical  Euclidean inner product} 
$g = \sum dx_i^2 +dy_j^2$, the corresponding
 \emph{Euclidean norm} $\|w \| = \sqrt{g(w,w)}$, the \emph{symplectic} and
 \emph{Liouville forms}
$$
\omega = \sum_{j=1}^n dx_j\wedge dy_j, \quad \lambda =\sum_{j=1}^n x_jdy_j
$$
with the property that $\omega=d\lambda$.

Let $\Sigma$ be a smooth manifold of real dimension $2$, also called a
 \emph{smooth surface}.
A smooth map $\ell:\Sigma\to\RR^{2n}$ is called \emph{isotropic} if $\omega$
vanishes along $\ell$, which is to say
$$
\ell^*\omega=0.
$$
If $\ell$ is an immersion (resp. an embedding), its image $S=\ell(\Sigma)$ is
called an \emph{immersed (resp. embedded) isotropic surface} of $\RR^{2n}$. In such
situation, we have necessarily $n\geq 2$. If $n=2$, which is the critical
dimension, then  $S$ is
called an immersed (resp. embedded) \emph{Lagrangian} surface of $\RR^4$.

\subsection{Piecewise smooth and piecewise linear maps}
\subsubsection{Some basic definitions}
We recall some standard definitions of piecewise linear topology.

A \emph{triangulation} 
of a smooth closed surface $\Sigma$ is a triple $\scrT= (\Sigma,\scrK,
\phi)$, where $\scrK$ is a finite simplicial complex contained in some
Euclidean space~$\RR^d$, together with a homeomorphism
$\phi:|\scrK|\to \Sigma$, where $|\scrK|$ denotes the topological subspace
of $\RR^d$ associated to
the simplicial complex~$\scrK$.
If the restriction  of  $\phi$ to every simplex $s$ of the triangulation $\scrK$ is a
\emph{smooth embedding}, we say that the triangulation $\scrT$ is a \emph{smooth
triangulation}.
Every surface triangulation in this paper is assumed to be smooth, even if
this is not stated.

Given  a triangulation $\scrT$ of $\Sigma$, any simplex $s$ of
$\scrK$ is diffeomorphic to  its image $\phi(s)\subset \Sigma$, by
definition. This
identification being understood, $s$ is also referred to as a simplex of
$\Sigma$, and we  drop all references to the map  $\phi$ for
simplicity of notations. 
As the simplicial complex $|\scrK|$ is homeomorphic to a surface, its
simplices
can only be  $0$, $1$ and $2$-dimensional. They are called respectively
\emph{vertices, edges} and \emph{facets} of the triangulation.

A real (or vector) valued \emph{piecewise smooth map} $h$ on $\Sigma$ is a continuous  map,
such that there
exists a smooth triangulation $\scrT$ of $\Sigma$, with the property that the restriction of 
$h$ to every
facet of the triangulation is smooth. Such a triangulation is said
to be \emph{adapted} to the piecewise smooth map~$h$.

Given a smooth triangulation $\scrT$ of $\Sigma$, we can define the class of \emph{piecewise
linear maps} on $\Sigma$. A real (or vector) valued map $h$ on $\Sigma$ is
piecewise linear with respect to $\scrT$, if it is continuous and has piecewise linear restriction on
each simplex of the triangulation. For a second triangulated surface
$\scrT_2=(\Sigma_2,\scrK_2,\phi_2)$, we say that a map
$h:\Sigma\to\Sigma_2$ is \emph{piecewise linear with respect to the
triangulations}
$\scrT$ and $\scrT_2$, if the composition $ \phi_2^{-1}\circ h:
\Sigma\to |\scrK_2|\subset \RR^{d_2}$ is piecewise linear in the previous
sense.
 A
\emph{piecewise linear isomorphism} between two triangulated surfaces is a
bijective piecewise linear map  between the triangulated surfaces. It turns
out that the inverse map of a piecewise linear isomorphism is also piecewise linear.

If two triangulations of $\Sigma$ are piecewise linear isomorphic, they
 define the same spaces of piecewise linear functions on $\Sigma$.
A \emph{piecewise linear structure} on a surface $\Sigma$ is an equivalence class
of triangulations, modulo piecewise linear isomorphisms. By definition, a
piecewise linear structure on $\Sigma$ has a well defined class of piecewise linear
functions.

\begin{rmks}
	By Cairns and Whitehead results~\cite{Cai35, Whi40}, smooth manifolds admit a unique
	piecewise linear structure, up to piecewise linear isomorphisms. 
	In other words, the \emph{smooth}
	triangulations used to
	define a piecewise linear structures on a smooth manifold exist.
	Furthermore, they are  all piecewise linear
	isomorphic and  locally standard. 

	According to these results,  every linear map  on $\Sigma$ is
	 understood with respect to  the essentially unique piecewise
	linear structure deduced from the smooth structure (like in
Theorem~\ref{theo:main} for instance).

Given a piecewise
linear map~$h$ on~$\Sigma$ with respect to some triangulation, we can always replace the triangulation by  a refinement, so that
the map~$h$ is affine on each simplex of the triangulation. This refinement
	operation does not change the associated piecewise linear
	structure. Such a
refined triangulation is said to be \emph{adapted} to $h$. 

	With the above terminology, a piecewise
	linear map is piecewise smooth with respect to some adapted
	triangulation.
\end{rmks}

Starting at \S\ref{sec:jaub}, the surface $\Sigma$ is  a quotient
torus $\EE/\Gamma$, where  $\Gamma$ is a
lattice of the Euclidean plane $\EE$. In this setting, it is convenient use another point of view for
triangulations. 
Suppose we are given  a $\Gamma$-invariant locally finite \emph{simplicial
decomposition} of $\EE$. Then there exists a  triangulation
$\scrT=(\EE/\Gamma,\scrK,\phi)$, unique up to simplicial isomorphism of
simplicial complex and the $\Gamma$-action,  with the following
properties~:
\begin{enumerate}
	\item For each simplex $s$ of $\scrK$, let $\tilde\phi:s\to\EE$ be
		a lift of
		map $\phi$ to the universal
		cover~$\EE$ of~$\Sigma$. This map is uniquely defined up to
		 the $\Gamma$-action. We require that~$\tilde \phi$
		is an affine map for every simplex $s$ of~$\scrK$.
	\item The  simplicial decomposition of~$\EE$ is recovered as the
		collection of simplices $\tilde \phi(s)$, where $s$ is one
		of the simplices of $\scrK$ and $\tilde
		\phi$ is one of the lifts of $\phi$.
\end{enumerate}
A locally finite $\Gamma$-invariant decomposition of $\EE$ provides a
collection of simplices covering $\EE$. This set of simpleces of $\EE$ is finite
modulo $\Gamma$ and can be seen as an \emph{abstract simplicial
complex}.
Then $\scrK$ is defined as a geometric realization of this abstract
simplicial complex, in some Euclidean space $\RR^d$. The existence and uniqueness property of a homeomorphism
$\phi:|\scrK|\to \Sigma$ with the above conditions is then essentially
trivial. 

In conclusion, triangulations of quotient tori can be defined thanks to
$\Gamma$-invariant locally finite  simplicial
decomposition of the universal cover $\EE$.

\subsubsection{Isotropic piecewise smooth maps}
Let $h:\Sigma\to\RR^{2n}$ 
be a piecewise smooth map and $\scrT$ a triangulation of $\Sigma$ adapted
to~$h$. The pull back $h^*\omega$ does not make sense \emph{a priori}, since the
tangent map to $h$ is not globally defined on $\Sigma$.
Yet, the restriction of the map $h$  to any facet $\face$ (i.e. a
$2$-dimensional simplex of the triangulation identified to its image in
$\Sigma$) of
$\scrT$ is smooth. In particular, the restriction of the pull-back
$h^*\omega|_\face$ along  a given facet of the triangulation makes sense. 
\begin{dfn}
	\label{dfn:psisotropic}
We say that a piecewise
smooth map $h:\Sigma\to\RR^{2n}$, endowed with an adapted triangulation, is
	isotropic, 
	if 
	$$h^*\omega|_\face=0
	$$
	for
every facet $\face$ of the triangulation.
\end{dfn}
\begin{rmk}
	\label{rmk:ae}
	Definition~\ref{dfn:psisotropic} does not depend 
	on the choice of an adapted triangulation, so that we can talk of
	piecewise smooth isotropic maps without any reference to an underlying
	triangulation.
	In fact, it is easy the see that for a piecewise smooth map
	$h:\Sigma\to\RR^{2n}$, the pullback $h^*\omega$ is defined almost
	everywhere. Furthermore $h$  is isotropic if, and only if, $h^*\omega= 0$
	almost everywhere, which does not involve the choice of an adapted
	triangulation.
\end{rmk}

\subsubsection{Norms for piecewise smooth maps}
 The usual
$C^0$-norm for maps $h:\Sigma\to \RR^{2n}$ is given by
$$
\|h\|_{C^0} = \sup_{x\in\Sigma}\|h(x)\|
$$
and more generally, we define
the norm on a domain $K\subset \Sigma$ 
by 
$$
\|h\|_{C^0(K)} = \sup_{x\in K}\|h(x)\|.
$$
If $\Sigma$ is assumed to be compact,  piecewise smooth maps  have finite $C^0$-norm
since they are continous.
However, $C^1$-norm of a piecewise smooth map
$h:\Sigma\to \RR^{2n}$ does not seem to be as standard as  the $C^0$-norm.
Lacking some standard classical reference, we ought to give our own definition.
For this purpose, let $\scrT$ be a triangulation of $\Sigma$ adapted to the piecewise
smooth map~$h$.
The restriction $h|_\face$ of $h$ to every facet $\face$ of $\scrT$ is smooth and
we can make sense of the norm
$$
 \|dh\|_{C^0(K\cap\face)} = \sup_{x\in K\cap \face} \|dh|_{\face,x}\|
$$
where the norm of $df$ is taken with respect to the operator norm induced
by the choice of a Riemannian metric $g_\Sigma$ on $\Sigma$ and
the Euclidean inner product $\|\cdot\|$ on $\RR^{2n}$. 
Using the notation $\fC_2(\scrT)$ for the set of facets of the triangulation, 
put 
\begin{equation}
	\label{eq:c0ps}
	\|dh\|_{C^0(K)}= \sup_{\face \in\fC_2(\scrT)}\|dh\|_{C^0(K\cap \face)}.
\end{equation}
The $C^1$-norm is eventually defined by
$$
\|h\|_{C^1( K)} = \|h\|_{C^0(K)} +  \|dh\|_{C^0(K)},
$$
and when $K=\Sigma$, we simply write the norm $\|h\|_{C^1}$.
\begin{rmk}
	An alternate definition could be given along the lines of
	Remark~\ref{rmk:ae}: since $h$ is piecewise smooth, its
	differential $dh$ makes sense almost everywhere. Thus the 
	$C^0$-norm of 
	$dh$ makes sense as an \emph{essential supremum}.
Therefore, this norm does not depend on the choice of triangulation
adapted to a piecewise smooth map $h$. Furthermore $\|h\|_{C^1}$ is 
finite for piecewise smooth maps if $\Sigma$ is compact.
\end{rmk}

\subsection{Embeddings and immersions}
Recall that a \emph{topological embedding}
$f:\Sigma\to \RR^{2n}$ is a continuous
map, such that the restriction $f:\Sigma\to f(\Sigma)$ is a homeomorphism,
where $f(\Sigma)\subset \RR^{2n}$
is endowed with the induced topology. We say that $f$ is a 
\emph{topological
immersion}, if $f$ is  locally a topological embedding. 
If $f$ is not assumed to be smooth, typically piecewise smooth or piecewise
linear, we will say that $f$ is an \emph{immersion} (resp. an
\emph{embedding}) if it is a \emph{topological immersion} (resp.
\emph{embedding}).

A standard application of the implicit function theorem shows that smooth
immersions are also topological immersions. 
There is an easy
characterization of  piecewise linear maps which are topological immersions
thanks to Lemma~\ref{lemma:inj}. Since this lemma is a classical result of
piecewise linear geometry and is not used in this paper, we do not give
any proof here.
\begin{lemma}
	\label{lemma:inj}
	A piecewise linear map $f:\Sigma\to \RR^{2n}$ is a topological
	immersion if, and only if, $f$ is locally  injective. If $\Sigma$
	is closed, and $f$ is globally injective, then it is a topological
	embedding.
\end{lemma}
The notion of \emph{polyhedral surface} was used
in Corollary~\ref{cor:main}, as well as the condition
 of a polyhedral surface being an approximation close to a smooth
surface, in the $C^1$-sense. The accurate definitions are  given below.
\begin{dfn}
	\label{def:close}
	A polyhedral surface,  $S_0 \subset \RR^{2n}$ is the 
	image of a smooth surface  $\Sigma$ by some piecewise linear map $h:\Sigma\to
	\RR^{2n}$.  

	If $h$ is a topological immersion (resp. embedding), then $S_0$ is
called an immersed (resp. embedded) \emph{polyhedral surface} of $\RR^{2n}$. 
	If $h$ is isotropic, in the sense of piecewise smooth isotropic
	maps, we say that the polyhedral surface $S_0$ is
	isotropic. In the case $n=2$, we say that $S_0$ is Lagrangian, instead
	of isotropic.

	For $\epsilon>0$, we say that a polyhedral surface $S_0$ of
	$\RR^{2n}$ is $\epsilon$-close to a
	smoothly immersed surface  $S$ of $\RR^{2n}$, in the
	$C^k$-sense (where $k=0$ or $1$), if there exists a smooth
	surface $\Sigma$ together with
	\begin{itemize}
		\item a smooth  immersion $f:\Sigma\to
			\RR^{2n}$ such that $f(\Sigma)=S$,
		\item a piecewise linear map $h:\Sigma\to\RR^{2n}$ such that
			$h(\Sigma)= S_0$,
\end{itemize}
with the property that
$$
	\|f-h\|_{C^k}\leq \epsilon,
$$
where the $C^1$-norm is defined using the Riemannian metric $g_\Sigma$ induced on $\Sigma$ by the canonical Euclidean
	structure of $\RR^{2n}$ and the immersion $f:\Sigma\to\RR^{2n}$.

	A sequence of polyhedral surfaces $S_N$ of $\RR^{2n}$, is called an
	approximation of a smooth surface $S$, in the $C^k$-sense, if there exists a sequence
	$\epsilon_N$ such that $S_N$ is $\epsilon_N$-close to $S$ in the
	$C^k$-sense for every
	$N$, and $\lim \epsilon_N=0$.
\end{dfn}
\begin{rmk}
The notion of $\epsilon$-closeness of Definition~\ref{def:close}
	depends on the choice of a Riemannian 
	metric $g_\Sigma$ on $\Sigma$. A choice has to
	be made for $g_\Sigma$. Here, the  metric $g_\Sigma$ 
	is chosen as the one induced by the ambient Euclidean space $\RR^{2n}$ and
	the choice of an immersion $f:\Sigma\to\RR^{2n}$. 
	Such a choice is somewhat arbitrary. Nevertheless, 
	this choice is not crucial 
	for our purpose, since
	another choice of Riemannian metric $g_\Sigma$ leads to
	an equivalent $C^1$-norm. Furthemore, the notion of approximation by
	a sequence of polyhedral surfaces does not depend on the choice of metric $g_\Sigma$.
\end{rmk}
The $C^1$-norm of piecewise  smooth maps is designed in such a way 
that various properties, which are classical in the  smooth case,
extend naturally in the piecewise smooth setting. 
The next proposition belongs to this category. Although the arguments are 
quite standard, a proof ought
to be given, since we rely on this result for the proof of our main theorem
and it seems impossible to find a clear reference for the proof.
\begin{prop}
	\label{prop:approx}
	Assume $\Sigma$ is a smooth closed 
	surface and $f:\Sigma\to \RR^{2n}$ is a
	smooth embedding (resp. immersion).  Then, there 
	exists $\epsilon>0$, such
	that for every piecewise smooth map $h:\Sigma\to \RR^{2n}$,
	the condition
	$$
	\|f-h\|_{C^1}\leq \epsilon
	$$
	implies that  $h$ is a
topological embedding (resp. immersion).
\end{prop}
\begin{proof}
The case of embeddings in the proposition readily follows from the case
of immersions exactly as in the smooth setting.
	Hence, we just prove the proposition in the case of a smooth immersion 
	$f:\Sigma\to\RR^{2n}$. 
	Let $\epsilon$ be a positive real number, to be fixed later on,
	and 
	$h:\Sigma\to\RR^{2n}$,  be a piecewise smooth map, such that
$$
\|f-h\|_{C^1}\leq \epsilon.
$$
	We introduce the affine space 
	subspace $F$ of $\RR^{2n}$ at a point $x$ of $\Sigma$, defined by
	$$
	F=f(x)+\Im\; df_x,
	$$
	understood as the image  of the tangent 
	space $T_x\Sigma$ by the tangent map~$f_*$. The affine map 
	$$
	\pi:\RR^{2n}\to F
	$$
	denotes the orthogonal projection onto $F$ and the
	corresponding \emph{projected maps} are given by
	$$
	\hat f = \pi\circ f \quad\mbox{ and }\quad  \hat h=\pi\circ h.
	$$
	Since $\pi$ is an orthogonal projection, its linear part $\vec\pi$
	satisfies
	$$\|\vec \pi\|= 1,$$
	which leads to the estimate
	\begin{equation}
		\label{eq:fpp}
		\|\hat f - \hat h\|_{C^1}\leq\|f-h\|_{C^1} \leq \epsilon.
	\end{equation}
Since $f$ is a smooth
	immersion, the map $\hat f:\Sigma\to F$ is a local  diffeomorphism
	 at $x$. Let $\scrT$ be a triangulation of the surface $\Sigma$
	adapted to the piecewise smooth map $h$. We introduce the closed
	neighborhood  $\overline U$ of $x$ corresponding to 
	the \emph{star} of $x$, which is the union of simplices of $\scrT$, that
	contain $x$. We denote by $\overline V\subset F$ its image  $\hat f(\overline U)$.
	Up to passing to a refinement of the triangulation
	$\scrT$, we can assume that $\overline U$ is an arbitrarily small
	closed neighborhood of $x$. In particular, 
we may assume that the restriction
		$\hat f:\overline U\to \overline V$, is a diffeomorphism
	with  inverse $\hat f^{-1}:\overline V\to \overline U$. 

	As the star of $x$, the domain
		$\overline U$ carries an induced triangulation, obtained as the
		retriction of $\scrT$. We deduce a triangulation $\scrT'$ of $\overline V$,
		obtained as the image triangulation of $\overline U$ by
		$\hat f$.
	By construction, the map
	$$\psi :\overline V\to F
	\quad \mbox{	defined  by }\quad
	\psi= \hat	h \circ \hat f ^{-1}
	$$
	is piecewise smooth, with $\scrT'$ as an adapted triangulation.

	The next lemma is a local inverse function theorem adapted to the
	case of the piecewise smooth map $\psi$.
	\begin{lemma}
		\label{lemma:diffeo}
		If $\epsilon>0$ is chosen sufficiently small, then for every
		 piecewise smooth map $h:\Sigma\to \RR^{2n}$ such that $\|f-h\|_{C^1}\leq
		\epsilon$ and every $x\in\Sigma$, 
		the corresponding  map $\psi:\overline V\to F$, introduced above, is a local homeomorphism near $f(x)$. 
	\end{lemma}
	\begin{proof}[End of the proof of Proposition~\ref{prop:approx}]
	Assuming Lemma~\ref{lemma:diffeo} holds, we may 
	replace $\overline V$ by a smaller open neighborhood $V'$ of $f(x)$ in $F$, so  that 
	the retriction	$\psi:V'\to \psi (V')$ is a homeomorphism.
	Hence,	there exists a continuous inverse
		$\varphi :\psi( V')\to  V'$ of $\psi:V'\to\psi(V')$.
		We define $U'$ as the open neighborhood $U'=\hat f^{-1}(V')$
		of $x$.
It follows that the map $\pi:h( U')\to \psi( V')$ is a
	homeomorphism onto its image. Indeed 
	$
	h\circ \hat f^ {-1}\circ \varphi: \psi( V') \to h( U') 
$
	is a continuous inverse of $\pi:h( U')\to \psi( V')$.
	Thus, $h( U')$ is
	homeomorphic to the open set $\psi( V')$
	of the vector space $F$. This shows
	that $h$ is a topological immersion, which completes the proof of
	Proposition~\ref{prop:approx}, in the case of an immersion.
	\end{proof}
\begin{proof}[Proof of Lemma~\ref{lemma:diffeo}]
	By Inequation~\eqref{eq:fpp}, we have the $C^0$-control
$$
	\|\psi (z) -z\|= \left \|\hat h(\hat f ^{-1}(z))- \hat f(\hat
	f^{-1}(z))\right \|\leq \epsilon,
$$
for every $z\in \overline V$, or, in other words
	\begin{equation}\label{eq:contr1}
	\|\psi -\id\|_{C^0(\overline V)}\leq \epsilon.
	\end{equation}
The differential of $\psi$ does not make sense everywhere, since the map is only
piecewise smooth with respect to the triangulation $\scrT'$ of $\overline
	V$.
	For every point
	$z\in \overline V$  contained in the interior of a facet of $\scrT'$, we have
	by~\eqref{eq:fpp}
	\begin{equation}
		\label{eq:fpp3}
		\left \|(d\hat h \circ d\hat f^{-1})_z -\id\right \|\leq
	\|\hat h - \hat f\|_{C^1}\|d\hat f^{-1}_z\|\leq \epsilon \|d\hat
	f^{-1}_z\|.
	\end{equation}
	The term $\|d\hat f^{-1}_z\|$ can be controlled uniformly thanks to
	the next lemma:
	\begin{lemma}
		\label{lemma:fpp}
There exists a constant $C>0$ such that for every $x\in \Sigma$, there
		exists a refinement of the triangulation $\scrT$, such that
		the diffeomorphism $\hat f: \overline U \to \overline V$ defined above satisfies
		$$
		\|d\hat f^{-1}_z\|\leq C
		$$
		for every $z\in \overline V$.
	\end{lemma}
	\begin{proof}
Since $f$ is an immersion and $\Sigma$ is compact, there exists a constant
		$C>0$, such that for every $x\in\Sigma$ and $u\in
		T_x\Sigma$, we have
$$
	\frac 2C \|u \|\leq 	\|df_x\cdot u\|.
$$
By definition $d\hat f = \vec \pi \circ df$, where $\vec \pi$ is the linear part of
		$\pi$. Furthermore  
		$$d\hat f_x=\vec \pi\circ df_x = df_x,$$
		since $\vec \pi$ is a projection onto $\Im\; df_x$.
		Hence,  for every tangent vector $w\in T_{f(x)}\overline V$,
$$
		\left \|d\hat f^{-1}_{f(x)}\cdot w \right \|\leq \frac C
		2\|w\|,
$$
i.e.
$$
		\left \|d\hat f^{-1}_{f(x)}\right \| \leq \frac C2.
$$
	Recall that $\overline U$ is the star of $x$ with respect to the
		triangulation $\scrT$. Up to passing to a refinement of the triangulation $\scrT$, we may
		assume that $\overline U$ is arbitrarily
		small, so that by the $C^1$-continuity  of $f$, 
		$$
		\|d\hat f_z^{-1}\|\leq C
		$$
		for every $z\in \overline V=\hat f(\overline U)$,
		which proves the lemma.
	\end{proof}
	In the rest of the proof of Lemma~\ref{lemma:diffeo}, 
	we  assume that a refinement of  $\scrT$ is chosen, so that
	the conclusion of Lemma~\ref{lemma:fpp} is met.
		By~\eqref{eq:fpp3} and Lemma~\ref{lemma:fpp}, we deduce
		that for every point $z$ of $\overline V$ contained in the interior
		of a facet of $\scrT'$, we have
	\begin{equation}\label{eq:contr2}
		\|d \psi_z - \id\|\leq \epsilon C.
	\end{equation}
		We choose $\epsilon =\frac 12\min\big (\frac 12
		, \frac 12 C^{-1}\big )$.
			Then by~\eqref{eq:contr1} and~\eqref{eq:contr2}
\begin{equation}
	\label{eq:contract0}
		\|\psi-\id\|_{C^1(\overline V)}\leq \frac 12.
\end{equation}
At this stage, it is natural
to consider the fixed points of the map 
	\begin{equation}\label{eq:contract}
		\Phi_y(z)= z- \psi(z) +y,
	\end{equation}
	defined for $y\in
	F$ and $z\in \overline V$. 
	\begin{lemma}
		The map $\Phi_y:\overline V\to F$ is a
		$\frac 12$-contracting map on
		a sufficiently small closed ball $\overline B_r$ of
		$\overline V$ centered at $f(x)$.
	\end{lemma}
	\begin{proof}
	This lemma relies on a version of the mean value theorem, in the context of
piecewise smooth functions.
		Obviously 
		\begin{equation}\label{eq:contract3}
		\Phi_y(z)-\Phi_y(z')= (\psi(z')-z')-(\psi(z)-z).
		\end{equation}
		We choose a closed Euclidean ball $\overline B_r\subset F$
		contained in $V$ and
		centered at $f(x)$ with radius $r>0$.
		For $z,z'\in\overline B_r$, the segment $[zz']$ is
		contained in $\overline B_r$ hence in $\overline V$.
	The mean value
		theorem  applies  to the
		function $\psi$. However some justification is needed
		since $\psi$  is only piecewise smooth:
		by transversality,
		the segment $[zz']$ can be approximated,
		in the $C^1$-sense, by 
		a smooth parametric curve 
		$\gamma:[0,1]\to \overline V$, such that
		 $\gamma(0)=z$, $\gamma(1)=z'$ and for every $0<t<1$
		\begin{itemize}
			\item
				$\gamma(t)$ does not go through a vertex of the
		triangulation $\scrT'$ of $\overline V$.
	\item Furthermore
		$\gamma(t)$ intersects transversaly the $1$-skeleton of the
		triangulation.
		\end{itemize}
		Such a curve $\gamma$ intersects the $1$-skeleton of
		$\scrT'$ only a finite number of
		times $0=t_0<t_1<\cdots <t_k=1$. Otherwise $\gamma(t)$ is
		contained in the interior domain of 
		facets. Then
		$$
		\Phi_y(\gamma(t_{j+1})) - \Phi_y(\gamma(t_j)) =
		\int_{t_j}^{t_{j+1}} d\Phi_y|_{\gamma(t)}\cdot \gamma'(t)
		\; dt.
		$$
		Since $d\Phi_y|_{\gamma(t)}=\id -d\psi|_{\gamma(t)} $  for every $t\in
		(t_j,t_{j+1})$, we deduce by~\eqref{eq:contract0} that
		$$
		\Big \|\Phi_y(\gamma(t_{j+1}))
		- \Phi_y(\gamma(t_j))\Big \| \leq 
		\frac {  L_j}2.
		$$
		where $L_j$ is the length of the curve $\gamma(t)$ from 
		$t_j$ to $t_{j+1}$.
		Adding all the identities together and the triangle
		inequality give the result
		$$
		\Big \|\Phi_y(z)-\Phi_y(z')\Big \|\leq \frac {L(\gamma)}2, 
		$$
		where $L(\gamma)$ is the length of $\gamma$.
		Passing to the
		limit with a sequence of curves $\gamma$ converging in the
		$C^1$-sense towards $[zz']$ 
gives the inequality
		$$
			\Big \|\Phi_y(z)-\Phi_y(z')\Big \|\leq \frac 12\|z-z'\|
			$$
			for every $z,z'\in\overline B_r$, which proves the
			lemma.
	\end{proof}
	The end of the proof of Lemma~\ref{lemma:diffeo} 
is based on the fixed point principle, applied to the map $\Phi_y$. Indeed,
a solution of $\Phi_y(z)=z$ satisfies $\psi(z)=y$, by definition, which
provides a construction for the inverse of $\psi$.
By \cite[Proposition
6.3.1]{JRT}, the  equation  $\Phi_y(z)=z$ admits a unique solution $z=\varphi(y)\in
\overline B_r$ for every $y\in \overline B_{r/2}$. This defines a
 local inverse $\varphi:\overline B_{\frac r2} 
\to \overline B_r$ of the map $\psi$. 

The fact that $\varphi$ is continuous is classical: for every $y,y'\in \overline
B_{r/2}$, we can write
\begin{align*}
	\|\varphi(y)-\varphi(y')\| & =
	\|z-z'\|\\
	&= \|\Phi_y(z)-\Phi_{y'}(z')\|\\
	&=
\|\Phi_y(z)-\Phi_y(z') + y-y'\|\\
	&\leq \frac 12\|z-z'\| +\|y-y'\|,
\end{align*}
which implies
$$
\frac 12\|\varphi(y)-\varphi(y') \|= \frac 12 \|z-z'\| \leq \|y-y'\|.
$$
Therefore $\varphi$ is a $2$-Lipschitz map  on the ball $\overline
B_{r/2}$. In particular $\varphi$ is continuous.
In conclusion  $\psi$ is a local
homeomorphism near $f(x)$. 
\end{proof}
This completes the proof of Proposition~\ref{prop:approx}.
\end{proof}
Proposition~\ref{prop:approx} applies in the special case  of piecewise
linear maps, which can be stated as the following corollary:
\begin{cor}
	\label{cor:approx}
	Let $\Sigma$ be a smooth closed surface and $f:\Sigma\to \RR^{2n}$
	a smooth immersion (resp. embedding).
	Any piecewise linear map $h:\Sigma\to\RR^{2n}$
	sufficiently close to $f$, in the $C^1$-sense,
	is  a topological immersion ( resp.  embedding). 

	Equivalently, any polyhedral surface sufficiently close, in the
	$C^1$-sense, to a smoothly immersed (resp. embedded) closed 
	surface, is an
	immersed (resp. embedded) polyhedral surface.
\end{cor}
\begin{rmk}
It is likely that an alternate proof of the above corollary could be given
	relying on  Lemma~\ref{lemma:inj}.
\end{rmk}

\subsection{Overview of the Jauberteau-Rollin-Tapie techniques}
\label{sec:jaub}
A leisurely introduction to
the crucial tools of~\cite{JRT} used in the proof of
Theorem~\ref{theo:main}  is given in this section. 
The interested reader is advised to keep the
article~\cite{JRT} within reach,  as a  detailed handbook about
 the
constructions that will only be sketched here.

\subsubsection{Conformal cover}
The Jauberteau-Rollin-Tapie construction starts from a smooth isotropic 
immersion (or embedding) 
$$
\ell:\Sigma\to\RR^{2n},
$$
where $\Sigma$ is a $2$-dimensional smooth torus. The Riemannian 
metric $g_\Sigma=\ell^* g$, induced on $\Sigma$ by
the canonical Euclidean metric $g$ of $\RR^{2n}$, is conformally flat, by the classical
uniformization theorem. Hence there
exists a covering map 
$$
p:\EE\to \Sigma
$$
from a $2$-dimensional Euclidean
affine space $\EE$ with Riemannian metric $g_\EE$, 
such that $p$  is a conformal map with respect to
$g_\EE$ and $g_\Sigma$.  The group of deck
transformations of $p$ is identified to a lattice $\Gamma$ of
the vector space $\overrightarrow {\EE}$ associated to the affine space 
$\EE$,
acting by translations. 
Thus, the flat quotient metric $g_\sigma$ induced by $g_\EE$ on $\Sigma$
and $g_\Sigma$ are in the same conformal class. More precicely, there
exists a smooth positive function $\theta:\Sigma \to \RR$, such that
$g_\Sigma=\theta g_\sigma$
and $p$  induces
a conformal  diffeomorphism
$$
\Sigma \simeq \EE/\Gamma.
$$
An oriented orthonormal frame 
$(O,\vec E_1,\vec E_2)$ of the affine space $\EE$, induces an affine isometry  
$$r:\RR^2\to
\EE,
$$
defined by $r(0)=O$, $r(e_1)=O+\vec E_1$ and $r(e_2)=O+\vec E_2$, where
$(e_1,e_2)$ is the canonical basis of $\RR^2$. A notion of
\emph{non degenerate pair} $(p\circ r,\ell)$ is introduced in~\cite[\S
5.3]{JRT}. This technical condition is important
for the application of the fixed point principle later on. 
However, non degeneracy can always be achieved 
 by~\cite[Proposition 5.3.3]{JRT}, modulo a suitable choice of
isometry~$r$. Thus, we could always assume that the non degeneracy
condition holds in the rest of our paper.

\subsubsection{Standard quadrangulations}
For every integer $N>0$, there is a standard lattice 
$\Lambda_N\subset \RR^2$ given by
$$\Lambda_N=\ZZ \frac {e_1}N \oplus \ZZ \frac {e_2}N.
$$
The lattice $\Lambda_N$ is understood as the set of vertices of the standard
\emph{quadrangulation} $\cQ_N(\RR^2)$ of $\RR^2$, tiled by squares of sidelength
$N^{-1}$, according to the Figure~\ref{figure:quad}, where
\begin{itemize}
	\item the vertices are denoted
		$\vert_{kl}=N^{-1}(ke_1+le_2)\in\Lambda_N$;
	\item $\face_{kl}$ are the facets of the quadrangulation.
\end{itemize}
\begin{figure}[H]
\begin{pspicture}[showgrid=false](-3,-3)(3,3)
                          \psscalebox{.8}
           {
             \psline (-3,3) (3,3)
\psline (-3,1) (3,1)
\psline (-3,-1) (3,-1)
\psline (-3,-3) (3,-3)

\psline (-3,3) (-3,-3)
\psline (-1,3) (-1,-3)
\psline (1,3) (1,-3)
\psline (3,3) (3,-3)


\psset{linecolor=red, linewidth=2pt,linestyle=solid, fillstyle=solid, fillcolor=red}
\rput(-2,-2){$\face_{k-1,l-1}$}
\rput(0,-2){$\face_{k,l-1}$}
\rput(2,-2){$\face_{k+1,l-1}$}

\rput(-2,0){$\face_{k-1,l}$}
\rput(0,0){$\face_{kl}$}
\rput(2.1,0){$\face_{k+1,l}$}

\rput(-2,+2){$\face_{k-1,l+1}$}
\rput(0,+2){$\face_{k,l+1}$}
\rput(2,+2){$\face_{k+1,l+1}$}

\color{red}
\pscircle(-1,-1){.1}
\rput[tr](-1.1,-1.1){$\vert_{k,l}$}
\pscircle(1,-1){.1}
\rput[tl](1.1,-1.1){$\vert_{k+1,l}$}

\pscircle(-1,+1){.1}
\rput[br](-1.1,1.1){$\vert_{k,l+1}$}
\pscircle(1,1){.1}
\rput[bl](1.1,1.1){$\vert_{k+1,l+1}$}



}
\end{pspicture}
\caption{Quadrangulation $\cQ_N(\RR^2)$}
\label{figure:quad}
\end{figure}

\subsubsection{Quotient quadrangulation}
The vector spaces $\RR^2$ is identified to $\overrightarrow \EE$ via the
linear isometry $\vec r$, which is the linear
part of the affine isometry $r:\RR^2\to\EE$. Hence, the 
lattices $\Lambda_N$ and $\Gamma$
can be understood as lattices of the same $2$-dimensional vector space
modulo the isomorphism $\vec r$.
If $\Gamma$ is a sublattice of 
$\Lambda_N$, we deduce that $\Sigma = \EE/\Gamma$ carries a
a  quotient quadrangulation $\cQ_N(\Sigma)$. Furthermore, the quotient
quadrangulation is acted on by  the residual action of $\Lambda_N$.
Unfortunately,  $\Gamma$ may 
not be a sublattice of $\Lambda_N$.
To get around this technical issue, 
we can make a suitable choice of a sequence of affine isomorphisms 
\begin{equation}\label{eq:ac}
	r_N:\RR^2\to\EE\quad \mbox{ such that }\quad r_N=r+\cO(N^{-1}),
\end{equation}
with the property that the image of $\Lambda_N$ by
$\vec r_N$ contains $\Gamma$ as a sublattice. The construction of the maps
$r_N$ is elementary and explained in details in~\cite[\S 3.2 ]{JRT}. 
\begin{rmk}
Although the maps $r_N$ are not isometric, they converge towards an
	isometry $r:\RR^2\to\EE$. In this sense, the maps $r_N$ are
	\emph{almost isometric}. It follows that the pull-back 
	metrics $ r_N^*g_\EE$ converge 
	towards the canonical Euclidean metric~$g_{\RR^2}$ and are uniformly
	commensurate with~$g_{\RR^2}$. In particular, any of these metrics could
	be used for uniform estimating purposes.
\end{rmk}

Given the almost isometric affine maps $r_N:\RR^2\to \EE$, the lattices
$\Lambda_N$ are now understood as lattice acting on $\EE$, with  $\Gamma$
acting as a sublattice of $\Lambda_N$.
We deduce that the quadrangulation $\cQ_N(\RR^2)$ descends as a quotient
quadrangulation of $\Sigma \simeq \EE/\Gamma$, denoted $\cQ_N$ or $\cQ_N(\Sigma)$.

\subsubsection{Quadrangular meshes}
\label{sec:quadmesh}
We
denote by $\fC_k(\cQ_N)$ the set  of $k$-cells of the quadrangulation
$\cQ_N$ and  by $\scrC^k(\cQ_N,X)$ the space
of functions on the set $\fC_k(\cQ_N)$, with values in some set $X$.
The vector space of \emph{quadrangular meshes}, denoted $\scrM_N$ or 
$\scrM(\cQ_N)$, on $\Sigma$ is defined by
$$
\scrM(\cQ_N)=\scrC^0(\cQ_N,\RR^{2n}).
$$
In other words, a quadrangular mesh $\tau\in\scrM_N$
is a $\RR^{2n}$-valued function on the space of the vertices of the quadrangulation $\cQ_N$.

Given a smooth map as $\ell:\Sigma\to \RR^{2n}$, there is a natural sequence
of \emph{approximations} of $\ell$  by quadrangular meshes
$$
\tau_N\in \scrM_N
$$
called the \emph{samples} of $\ell$ and
defined by 
$$
 \tau_N(\vert) = \ell(\vert),
$$
for every vertex $\vert$ of the quadrangulation~$\cQ_N$.
\begin{rmk}
If $\ell$ is isotropic, it is not clear at all whether and, in which sense,
	the samples  $\tau_N$ are isotropic. We are going to give a definition of isotropic quadrangular
	meshes, and show that the samples $\tau_N$ are
	almost isotropic.
\end{rmk}
\subsubsection{Isotropic quadrangular meshes}
Let $D$ be a compact
oriented surface of $\RR^{2n}$ with boundary $\del D$. By Stokes theorem
$$
\int _D\omega = \int_{\del D}\lambda,
$$
where $\lambda$ is the Liouville form. If $D$ is isotropic, the
integral of the Liouville form along it boundary $\del D$ vanishes.

Given a quadrangular mesh $\tau\in\scrM_N$, every facet $\face_{kl}$ of the
quadrangulation has an associated quadrilateral of $\RR^{2n}$, given by the
points $A_j=A_j(\tau,\face_{kl})$ for  $0\leq j\leq 3$, with
$$
A_0= \tau(\vert_{kl}),\quad
A_1= \tau(\vert_{k+1,l}),\quad
A_2= \tau(\vert_{k+1,l+1}),\quad
A_3= \tau(\vert_{k,l+1}).
$$
The quadrilateral $(A_0A_1A_2A_3)$ is called \emph{the quadrilateral of the
mesh $\tau$, along the facet
$\face_{kl}$}. 
This motivates the following
definition:
\begin{dfn}
	\label{dfn:quad}
	An oriented closed piecewise smooth curve $\gamma$ of $\RR^{2n}$ is called  isotropic
	if $\int_\gamma\lambda =0$. 

	In particular if $\gamma$ is a parametrization of
	the four edges of a quadrilateral, we say that the quadrilateral is
	isotropic.

	A quadrangular mesh $\tau\in\scrM_N$ 
	is called isotropic, if every quadrilateral of $\RR^{2n}$ 
	associated to a facet of the quadrangular mesh $\tau$ is isotropic.
\end{dfn}

Isotropic quadrangular meshes are solutions of an obvious system of
 quadratic equations, interpreted as the vanishing of a 
 \emph{discrete symplectic density}.
Given $\tau\in\scrM_N$ or a smooth function $f$ on $\Sigma$, it is often
convenient to work on the cover $p_N= p\circ r_N:\RR^2\to \Sigma$. The pull
back $\tilde f = f\circ p_N$ of $f$ is understood as a smooth
$\Gamma$-periodic function of $\RR^2$. Similarly, $\tilde \tau=\tau\circ
p_N$ is understood as a $\Gamma$-invariant quadrangular mesh of
$\cM(\cQ_N(\RR^2))=\scrC^0(\cQ_N(\RR^2),\RR^{2n})$. Most of the time, we
will work on the cover $\RR^2$ without any warning or special notations.

We then introduce the \emph{vectors fields}  
$$\scrU_\tau \quad \mbox{ and } \quad \scrV_\tau
\in\scrC^2(\cQ_N(\Sigma),\RR^{2n})
$$
defined on the cover by the formulas
$$
\scrU_\tau(\face_{kl}) = \frac N{\sqrt 2} 
\left (\tilde\tau(\vert_{k+1,l+1}) 
- \tilde\tau(\vert_{kl}) \right )
$$
and
$$
\scrV_\tau(\face_{kl}) = \frac N{\sqrt 2} 
\left (\tilde\tau(\vert_{k,l+1}) 
- \tilde\tau(\vert_{k+1,l}) \right )
$$
\begin{rmk}
The above vectors are precisely the diagonals of the quadrilateral  of
$\RR^{2n}$, 
given by the facet of $\tau$ associated to $\face_{kl}$
and renormalized by a factor~$\frac
N{\sqrt 2}$.
Thus $\scrU_\tau$ and $\scrV_\tau$ can be understood as finite difference 
 versions
of partial derivatives of $\tau$ in diagonal directions.
\end{rmk}

An easy calculation (cf.~\cite[\S 4.1]{JRT}) shows that the integral of the Liouville form along
the quadrilateral associated to the quadrangular mesh $\tau$ and the facet
$\face_{kl}$ is precisely 
$$
N^{-2}\omega(\scrU_\tau(\face_{kl}), \scrV_\tau(\face_{kl})).
$$
Notice that $N^{-2}$ is the Euclidean area of the facets of $\cQ_N(\RR^2)$. This
motivates the definition of the symplectic density
$$
\mu_N:\scrM_N\to \scrC^2(\cQ_N,\RR)
$$
given by
\begin{equation}
	\label{eq:density}
	\mu_N(\tau)  = \omega(\scrU_\tau,\scrV_\tau).
\end{equation}
By definition $\mu_N^{-1}(0)$ is precisely the set of isotropic
quadrangular meshes.

\subsubsection{Norms for quadrangular meshes}
\label{sec:norms}
Certain zeroes of $\mu_N$ were constructed in~\cite{JRT}, thanks to the fixed point
principle. Some adapted norms have to be introduced in order to carry out
the construction.
 There are two special \emph{diagonal} translations $T_u$ and $T_v$ acting
on facets or vertices of $\cQ_N(\RR^2)$ (or $\cQ_N(\Sigma)$), given by
$$
T_u(\vert_{kl}) = \vert _{k+1,l+1} \quad \mbox{ and } \quad
T_v(\vert_{k+1,l})= \vert _{k,l+1}
$$
with similar formulas for facets.
These operators induce finite difference operators on the space of
functions $\scrC^0$ or $\scrC^2$. For instance, we define for a
quadrangular mesh $\tau$
\begin{equation}\label{eq:fd}
	\frac {\del  \tau}{\del \vec u} = \frac N{\sqrt 2}(\tau\circ T_u -
 \tau)
\quad \mbox { and } \quad
\frac {\del  \tau}{\del \vec v} = \frac N{\sqrt 2}(\tau\circ T_v -
 \tau).
\end{equation}
Such finite differences, are analogues of the derivatives in
$(u,v)$-coordinates obtained by rotating the canonical basis $(e_1,e_2)$ of $\RR^2$ by
an angle $+\frac \pi 4$. More precisely by setting
\begin{equation}\label{eq:uv}
u= \frac{x+y}{\sqrt 2}, \quad  v=\frac {y-x}{\sqrt 2}.
\end{equation}
The finite differences induce $\cC^k_w$-norms on the space of 
functions $\scrC^j(\cQ_N)$ on the set of $j$-cells of $\cQ_N$. 
These norms are 
discrete versions of the $C^k$-norms for smooth functions (notice the
slightly different typography for the norms in the smooth and discrete cases). 
The $\cC^k_w$-norms are denoted 
$\|\cdot \|_{\cC^k_w}$, where $w$ stands for \emph{weak}. 
Indeed,  the
translations $T_u$ and $T_v$  only span  a sublattice of index~$2$ of
$\Lambda_N$. It follows that these weak norms do not control every finite
differences,  in particular in the directions  
\begin{equation}\label{eq:k}
	k_1=N^{-1}{e_1} \quad \mbox{ and }\quad
k_2=N^{-1}{e_2}.
\end{equation}
It is also possible to mimic  Hölder norms, with Hölder regularity
$\alpha\in(0,1)$ in the $T_u$ and $T_v$ directions. The corresponding
$\cC^{k,\alpha}_w$-norms
are denoted $\|\cdot \|_{\cC^{k,\alpha}_w}$. 
\begin{rmk}
More details about 
discrete weak norms and their properties
	can be found in \cite[\S 3.7]{JRT}. The reader may
	still wonder why such weak norms must be used here. Although we do
	not have a mathematical argument, there is some geometrical
	evidence,
	which tends to show why 
	weak norms are natural, and why they are the best one can expect for
	the discrete analysis of $\mu_N$.

	First, the space of solutions of the equation $\mu_N=0$ has 
	symmetries
	obtained from the \emph{shear action}, described in~\cite[\S 4.2]{JRT}.
	The shear action essentially pulls apart  two intertwined submeshes of
	a  given quadrangular mesh, by independent translations along two
	index $2$ sublattices of $\Lambda_N$. The shear action leaves
	the equation
	invariant.  This seems to indicate that  strong $C^1$ controls cannot
	be expected from the mere equations $\mu_N=0$ (cf. \cite[Figure 1]{JRT}).

	The second observation concerns the 
	 linearization of the equation $\mu_N=0$. For the fixed point
	 principle, it is convenient to
	 consider  variations of $\mu_N$ in some particular directions,
	 given by  discrete functions
	 on
	 $\Sigma$ (cf. \cite[Chapter 4]{JRT}). This idea is motivated 
	 by  a moment map interpretation of
	 the smooth equation (cf. \cite[Chapter 2]{JRT}), 
	 where we are looking for perturbations along
	 complexified orbits of the gauge group.
	Eventually, such perturbations are given by
	 functions (for the smooth  or discrete cases) on $\Sigma$, and the linearized
	 operator  is essentially a Laplacian~(cf.
	\cite[\S 2.3.3 and \S 4.6.4]{JRT}). However,  the
	sequence of discrete Laplacians does not converge towards the
	operator in the smooth setting, as $N$ goes to infinity. Instead,
	the sequence of discrete Laplacians
	converges towards
 another Fredholm operator, acting 
	on  pairs
	of smooth functions, 
	rather than  single functions (see \S\cite[Theorem
	5.1.1]{JRT}). This phenomenon may be regarded as some
	  infinitesimal  reminiscence
	of the shear action. 
 Hence, natural norms for our problem should 
	allow sequences of solutions of the linearized discrete problem to
	converge towards pairs of functions for the limiting Fredholm operator on
	the smooth surface.
	This feature is  the essence of the
	weak norms (cf. \cite[\S 3.8 and Example 3.8.4]{JRT}).

However, it is fairly
	possible that  simpler discretization schemes, other than
	quadrangulations, could be used. Of course, any relevant 
	suggestion would be
	most welcome.
\end{rmk}
\subsubsection{Existence of isotropic meshes}
We can quote the main technical tool for producing isotropic quadrangular
meshes:

\begin{theo}[\cite{JRT}]
	\label{theo:jrt}
	Given a smooth isotropic immersion $\ell:\Sigma\to \RR^{2n}$,
	there exists a sequence of isotropic quadrangular meshes
	$\rho_N\in\scrM_N$, such that
	$$
	\|\rho_N-\tau_N\|_{\cC^{0}} = \cO\left 
	(\|\mu_N(\tau_N)\|_{\cC^{0,\alpha}_w} \right )
	$$
	where $\tau_N$ are the samples of $\ell$.
\end{theo}
\begin{proof}
	Although Theorem~\ref{theo:jrt} is essentially proved in
	 \cite[Theorem 6.3.2
	and Proposition
	6.3.3]{JRT}, the statement is slightly different.  In particular, 
	the estimate  $ \cO\left 
	(\|\mu_N(\tau_N)\|_{C^{0,\alpha}_w} \right )$ is  replaced
	with  a rough $\cO(N^{-1})$ estimate. 

	Nevertheless, the construction of $\rho_N$ proceeds from the fixed
	point principle applied to a certain map $T_N$ defined in \cite[\S
	6]{JRT}. The solution $\phi_N=T_N(\phi_N)$ is controlled in
	$\cC^{2,\alpha}_w$-norm by the term  $T_N(0)=-G_N(\mu_N(\tau_N))$,
	where $G_N$ is the Green operator of the discrete Laplacian, defined
	in
	\cite[\S 6.1]{JRT}.
	By \cite[Proposition 6.1.2]{JRT}, the $\cC^{0,\alpha}_w$-norm of
	$\mu_N(\tau_N)$ controls the $\cC^{2,\alpha}_w$-norm of
	$G_N(\mu_N(\tau_N))$. Eventually, the $\cC^{2,\alpha}_w$-norm of
	$\phi_N$ controls the $\cC^0$-norm of $\rho_N-\tau_N$, since $\rho_N=
	\tau_N-J\delta^\star_N\phi_N$, as defined by \cite[Proposition
	6.3.3]{JRT}, and the theorem follows.
\end{proof}
\begin{rmk}
	The key to obtain a $C^1$-control, as in Theorem~\ref{theo:main}, is to
improve the estimate of the \emph{error term} $\mu_N(\tau_N)$.
	This is precisely what is done in Proposition~\ref{prop:betterest},
	with  a $\cO(N^{-2})$-estimate, rather than a
	$\cO(N^{-1})$-estimate as in~\cite{JRT}.
\end{rmk}
\subsubsection{From quadrangulations to triangulations}
\label{sec:quadtri}
Quadrangular meshes do not provide piecewise linear maps in an obvious way.
They have to be completed into triangular meshes beforehand.
Indeed, a quadrilateral of $\RR^{2n}$ is not necessarily contained in a
$2$-plane. The idea is to fill  a quadrilateral $(ABCD)$ with a pyramid as in
Figure~\ref{figure:pyramid}.

\begin{figure}[H]
  \begin{pspicture}[showgrid=false](-2,-1)(2,1)
    \psline(-1,-1)(0,-1)(.5,0)(-0.5,0)(-1,-1)
    \psline(.5,1)(-1,-1)(0,-1)(.5,0)(-0.5,0)
    \psline(.5,1)(-1,-1)
    \psline(.5,1)(0,-1)
    \psline(.5,1)(.5,0)
    \psline(.5,1)(-0.5,0)

    \rput(-1.5,-1){$A$}
    \rput(0.5,-1){$B$}
    \rput(1.0,0){$C$}
    \rput(-1.0,0){$D$}

    \rput(0.1,1.1){$P$}

   \psset{fillstyle=solid,fillcolor=black}
   \pscircle(-1,-1){.1}
   \pscircle(0.5,0){.1}
   \pscircle(-0.5,0){.1}
   \pscircle(0,-1){.1}
   \pscircle(0.5,1){.1}
   
  \end{pspicture}
  \caption{Pyramid with apex $P$ and base $(ABCD)$}
  \label{figure:pyramid}
\end{figure}
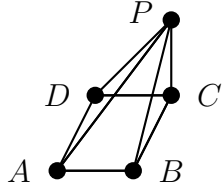
\begin{dfn}
A pyramid is called isotropic if its four triangles from the apex $P$ are
contained in isotropic planes. 
\end{dfn}
Given a quadrilateral $(ABCD)$ of $\RR^{2n}$, we are looking for a point
$P\in\RR^{2n}$, such that the corresponding pyramid is isotropic. This
equation, bearing on $P$, is a linear system. 
The compatibility condition is
precisely given by the fact that the quadrilateral is isotropic, in the sense of
Definition~\ref{dfn:quad}. The dimension of the affine space spanned by 
$(ABCD)$ is called the dimension of the quadrilateral.
The
dimension of the quadrilateral is exactly the codimension of the space of
solutions $P$, such that the pyramid is isotropic. Generic
quadrilaterals are $3$-dimensional, but they may be flat or more degenerate as
well.

\begin{dfn}
	\label{dfn:optimal}
Let $G$ be the barycenter of an isotropic quadrilateral $(ABCD)$ of $\RR^{2n}$.
The  closest  apex $P$ to
$G$, such that the corresponding pyramid is isotropic is called the
\emph{optimal apex} of $(ABCD)$.
\end{dfn}

Following this idea, we refine the quadrangulations $\cQ_N(\RR^2)$ into
triangulations $\scrT_N(\RR^2)$. In this context, we use the definition of a
triangulation as a simplicial decomposition of $\RR^2$ invariant by the
lattice. To do this, we replace each facet $\face_{kl}$ of the
quadrangulation with a vertex $\zert_{kl}$ (represented by a black dot) at its barycenter and complete
with four edges and 
four facets, according  to Figure~\ref{figure:triangulation}.
  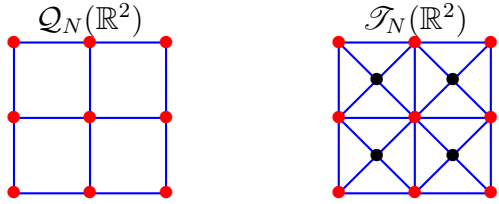
\begin{figure}[H]
  \begin{pspicture}[showgrid=false](-1,-1.5)(1,1.5)
    \psset{linecolor=blue }
    \psline (-1,1) (1,1)
    \psline (-1,0) (1,0)
    \psline (-1,-1) (1,-1)
    \psline (-1,1) (-1,-1)
    \psline (0,1) (0,-1)
    \psline (1,1) (1,-1)
    \color{red}
      \rput (-1,1){$\bullet$}
      \rput (-1,0){$\bullet$}
      \rput (-1,-1){$\bullet$}
      \rput (0,1){$\bullet$}
      \rput (0,0){$\bullet$}
      \rput (0,-1){$\bullet$}
      \rput (1,1){$\bullet$}
      \rput (1,0){$\bullet$}
      \rput (1,-1){$\bullet$}
      \color{black}
      \rput (0,1.3){$\cQ_N(\RR^2)$}
  \end{pspicture}
  \hspace{2cm}
  \begin{pspicture}[showgrid=false](-1,-1.5)(1,1.5)
    \psset{linecolor=blue }
    \psline (-1,1) (1,1)
    \psline (-1,0) (1,0)
    \psline (-1,-1) (1,-1)
    \psline (-1,1) (-1,-1)
    \psline (0,1) (0,-1)
    \psline (1,1) (1,-1)
    \psline (-1,1) (1,-1)
    \psline (0,1) (1,0)
    \psline (-1,0) (0,-1)
    \psline (-1,-1) (1,1)
    \psline (-1,0) (0,1)
    \psline (0,-1) (1,0)
    \color{red}
      \rput (-1,1){$\bullet$}
      \rput (-1,0){$\bullet$}
      \rput (-1,-1){$\bullet$}
      \rput (0,1){$\bullet$}
      \rput (0,0){$\bullet$}
      \rput (0,-1){$\bullet$}
      \rput (1,1){$\bullet$}
      \rput (1,0){$\bullet$}
      \rput (1,-1){$\bullet$}
	  \color{black}
	  \rput (-.5,.5){$\bullet$}
      \rput (.5,.5){$\bullet$}
      \rput (-.5,-.5){$\bullet$}
      \rput (.5,-.5){$\bullet$}
      \color{black}
      \rput (0,1.3){$\scrT_N(\RR^2)$}
  \end{pspicture}
  \caption{Triangular refinement of a quadrangulation}
	  \label{figure:triangulation}
  \end{figure}
Like quadrangulations, the triangulations $\scrT_N(\RR^2)$ descend to the
quotient $\Sigma$,
via the covering map $p\circ r_N:\RR^2\to \Sigma$. The quotient
triangulation on $\Sigma$ is  denoted $\scrT_N(\Sigma)$
or simply~$\scrT_N$. The space of triangular meshes $\scrM'_N$, analogous to
quadrangular meshes $\scrM_N$, is defined by
$$
\scrM'_N = \scrC^0(\scrT_N,\RR^{2n}).
$$
The samples $\tau_N\in\scrM_N$ of $\ell$ can be extended as samples
$\tau'_N\in\scrM'_N$, defined by
$$
\tau'_N(\vert)= \ell(\vert)
$$
for every vertex $\vert\in\fC_0(\scrT_N)$.
Similarly the isotropic quadrangular meshes $\rho_N$ of
Theorem~\ref{theo:jrt} can be extended in two flavors
$$\hat\rho_N,\rho'_N \in\scrM'_N.
$$
The former triangular mesh is defined by
$$
\hat\rho_N(\vert) =\rho'_N(\vert) = \rho_N(\vert)
$$
for every $\vert \in\fC_0(\cQ_N)$ (i.e. one of the red vertices in
Figure~\ref{figure:triangulation}).
If $\zert_{kl}$ is a vertex of
$\fC_0(\scrT_N)$ which does not belong to $\fC_0(\cQ_N)$ (i.e. one of the
black vertices in Figure~\ref{figure:triangulation}), we define
$\hat\rho(\zert_{kl})$ as the barycenter of the quadrilateral associated to
the facet $\face_{kl}$ of $\rho_N$. More explicitely
$$\hat \rho_N(\zert_{kl}) = \frac 14 \sum_{\vert \in \zert_{kl}}\rho_N(\vert).
$$

The latter triangular mesh $\rho'_N$ is defined as follows:
the quadrilateral associated to the facet $\face_{kl}$ of $\rho_N$
is isotropic, by definition of $\rho_N$.
 Let $P_{kl}\in\RR^{2n}$ 
 be the optimal apex of this quadrilateral, in the sense of
Definition~\ref{dfn:optimal}. We then define 
$$
\rho'_N(\zert_{kl})=P_{kl}.
$$
\subsubsection{From triangular meshes to piecewise linear maps}
A triangular mesh  $\rho'\in\scrM'_N$ defines a 
natural piecewise linear map. Indeed an affine map along a Euclidean
triangle is defined by the values of the map at the three vertices.
In particular, the piecewise linear maps $\ell_N$ associated to $\rho_N'$
are
isotropic, by contruction. Furthermore, 
we will see that the maps $\ell_N$ provide the solutions to
Theorem~\ref{theo:main}. This motivates the following definition.
\begin{dfn}
	\label{dfn:optpl}
	The isotropic 
	piecewise linear map $\ell_N:\Sigma\to\RR^{2n}$, defined by the
	triangular mesh
	$\rho'_N:\Sigma\to\RR^{2n}$ constructed above, is called the
	$N$-th optimal piecewise linear
	isotropic approximation of $\ell$. 
\end{dfn}

\section{Proofs}

\subsection{Proof of Corollaries}

\begin{proof}[Proof of Corollary~\ref{cor:main}]
	Let $S$ be a smooth immersed (resp. embedded) isotropic surface of 
	$\RR^{2n}$ diffeomorphic to a $2$-torus $\Sigma$. 
Hence, there exists a smooth isotropic immersion (resp.
	embedding) $\ell:\Sigma\to \RR^{2n}$ such that $\ell(\Sigma)=S$. 
	By Theorem~\ref{theo:main}, given any $\epsilon>0$, there exists a piecewise
	linear isotropic map $f:\Sigma\to \RR^{2n}$, with $\|\ell- f\|_{C^1}\leq
	\epsilon$, which is 
	a topological immersion (resp. embedding).
	By definition~$S_0=f(\Sigma)$ is an
	immersed (resp. embedded) polyhedral isotropic surface,
	$\epsilon$-close to $S$ in the $C^1$-sense.
\end{proof}

\begin{proof}[Proof of Corollary~\ref{cor:hprin}]
	We start with an immersed (resp. embedded) torus $S$ of $\RR^{2n}$, where
	$n\geq 2$. By definition, there exists a smooth torus $\Sigma$ and
	a smooth immersion
	$f:\Sigma\to \RR^{2n}$, such that $f(\Sigma)= S$. By
	a result of Gromov, for every $\epsilon>0$, there exists a smooth
	isotropic immersion $\ell:\Sigma\to \RR^{2n}$, such that
	$\|f-\ell\|_{C^0}\leq \epsilon$. Furthermore, if $n\geq 3$ and $f$ is an
	embedding, we may assume that $\ell$ is an isotropic embedding.
	As in the proof of Corollary~\ref{cor:main}, we can approximate $\ell$  by a
	piecewise linear isotropic topological immersion (or embedding if
	$n\geq 3$ and $S$ is embedded) $h:\Sigma\to \RR^{2n}$, 
	such that $\|\ell-h\|_{C^1}\leq \epsilon$. We deduce that
	$\|f-h\|_{C^0}\leq 2\epsilon$ so that the isotropic immersed
	polyhedral surface $S_0=h(\Sigma)$
	is $2\epsilon$-close to $S$, in the $C^0$-sense.
	This proves the corollary.
\end{proof}

\subsection{Symplectic density estimate}
The notation $\cO(N^{-q})$, for some integer $q$, is going be used a lot in all the following
estimates. Formally, $F(N,\nu_1,\nu_2,\cdots)=\cO(N^{-q})$, where $F$ is an
application taking values in some Euclidean vector space, means that there
exist a constant $C>0$, such that 
$$
\|F(N,\nu_1,\cdots)\|\leq CN^{-q}
$$
for every integer $N\geq 0$, and every values of the parameters
$\nu_1,\cdots$ etc\dots In our case, the parameters are typically indices
$(k,l)$ and
points in $\Sigma$. 
Thus the notation $\cO(N^{-q})$ is always understood for some uniform constant
$C>0$, which depends only on the choice of the smooth map $\ell$, or, more precisely, on
its $C^k$-norm.
The $\cO(N^{-q})$ will  appear in approximations given by the Taylor
expansion of $\ell$. We use the  notation $d^q\ell_z$ for the $q$-th
differential of $\ell$ at $z\in\Sigma$. The $q$-th differential is a
$q$-multilinear form, but, for practical reasons, we shall use the notation  
$$d^q\ell_z\cdot\zeta=
d^q\ell_z(\zeta,\cdots,\zeta)$$
when the same tangent vector $\zeta$ is
repeated $q$ times.

The key for the proof of Theorem~\ref{theo:main} is a sharp estimate of the
error term, essentially the symplectic density, involved in Theorem~\ref{theo:jrt}, 
as stated in the next proposition.
\begin{prop}
	\label{prop:betterest}
	Given a smooth isotropic map $\ell:\Sigma\to\RR^{2n}$ and its
	sequence of samples $\tau_N\in\scrM_N$ as defined in
	\S\ref{sec:quadmesh}, we have
	the estimates
	$$
	\|\mu_N(\tau_N)\|_{\cC^{1}_w}=\cO(N ^{-2}). 
$$
and
	$$
	\|\mu_N(\tau_N)\|_{\cC^{0,\alpha}_w}=\cO(N ^{-2}). 
$$
\end{prop}
\begin{proof}
	The second estimate follows from the first according to the
	following lemma.
	\begin{lemma}
		The $\cC^{1}_w$-norm controls the $\cC^{0,\alpha}_w$-norm.
		In other words, there exists a constant $C>0$, such that
		for every $N>0$ and every 
		function $\phi\in\scrC^2(\cQ_N(\Sigma))$, we have
		$$
		\|\phi\|_{\cC^{0,\alpha}_w}\leq C 
		\|\phi\|_{\cC^{1}_w}.
		$$
	\end{lemma}
	\begin{proof}
This control is obtained in a similar way to the smooth case, where a
		control on first derivatives induces a control on Hölder
		regularity. We have not really given
		the definition of the weak Hölder norms in this paper and
		the reader should refer to \cite[\S 3.7]{JRT} for more
		details.
		Let $\phi$ be a function defined on  the set of facets of
		$\cQ_N(\RR^2)$. If $\face_1$ and $\face_2$ are two facets,
		their distance is 
		$d(\face_1,\face_2)=\|\zert_1-\zert_2\|$ by convention, 
		where $\zert_j$
		are the barycenters of $\face_j$. 
		If $\face_2$ is obtained by a diagonal translation of
		$\face_1$, then  $\face_2=T_u^k\circ T^l_v(\face_1)$, where $T_u$ and
		$T_v$ are the diagonal translations introduced at
		\S\ref{sec:norms}. 
		By definition, the finite difference satisfy
		$$
		|\phi(\face)-\phi(T_u\face )| \leq 
		\frac{\sqrt 2}N \left \|\frac {\del\phi}
		{\del \vec u}\right \|_{\cC^0},
		$$
		with a similar inequality for the $v$-finite difference.
		We deduce
		$$
		|\phi(\face_1) - \phi(\face_2)|
		\leq \sqrt 2 \frac{|k|+|l|}N \|\phi\|_{\cC^1_w}.
		$$
		We deduce that
		$$
		|\phi(\face_1) - \phi(\face_2)|
		\leq   2 d(\face_1,\face_2)\|\phi\|_{\cC^1_w}.
		$$
		If $0<d(\face_1,\face_2)\leq 1$, we have
		$$
		\frac{|\phi(\face_1) -
		\phi(\face_2)|}{d(\face_1,\face_2)^\alpha}
		\leq  2 \|\phi\|_{\cC^1_w}.
		$$
		If $d(\face_1,\face_2)>1$, we have

		$$
		\frac{|\phi(\face_1) -
		\phi(\face_2)|}{d(\face_1,\face_2)^\alpha}
		\leq	|\phi(\face_1) -\phi(\face_2)| 
		\leq 2\|\phi\|_{\cC^0}
		$$
		which proves the estimate of the lemma in the case of
		$\scrC^2(\cQ_N(\RR^2))$. The case of $\scrC^2(\Sigma)$
		endowed with its $\cC^1_w$ and $\cC^{0,\alpha}_w$-norms
		follows immediately, since these norms are obtained by
		passing to the cover $p\circ r_N:\RR^2\to\Sigma$. 
	\end{proof}
	We return to the proof of Proposition~\ref{prop:betterest} and we
	use the notation
	$$\eta_N=\mu_N(\tau_N)$$ 
	as a shorthand for the rest of the proof.
By definition 
	$$
	\eta_N(\face) =  \omega(\scrU_{\tau_N}(\face),\scrV_{\tau_N}(\face)).
$$
Using the index notations of \cite[\S 3.3.1 and \S 4.1.4]{JRT} 
$$
D^u_{ij}= \frac {\sqrt 2}N \scrU_{\tau_N}(\face_{ij})
\quad
D^v_{ij}= \frac {\sqrt 2}N \scrV_{\tau_N}(\face_{ij})
$$
identifying all
the functions with their lift
via $p\circ r_N$, we have
$$
 \frac 2{N^{2}} \eta_N(\face_{ij})=\omega(D^u_{ij},D^v_{ij})
$$
where, by definition
$$
D^u_{ij}= \ell (\vert_{i+1,j+1}) - \ell (\vert_{i,j})
$$
and
$$
D^v_{ij}= \ell(\vert_{i,j+1}) - \ell (\vert_{i+1,j})
$$
We can use the Taylor formula at the center $\zert_{ij}$ of the facet $\face_{ij}$
for the function~$\ell$. We are going to use the notation 
$$  \xi =
\frac 1{2N} (e_1+e_2)=\frac 12(k_1+k_2),\quad 
\zeta = \frac 1{2N}(e_2-e_1)=\frac 12(k_2-k_1), 
$$
and the $(u,v)$ coordinates introduced with Formula~ \ref{eq:uv}, obtained by rotating
the canonical basis $(e_1,e_2)$ of $\RR^2$ by an angle $+\frac\pi 4$.
In particular
$$
\xi = \frac {1}{\sqrt 2 N}\frac\del{\del u},\quad
\zeta = \frac {1}{\sqrt 2 N}\frac\del{\del v}.
$$
Then
$$
D^u_{ij} = \ell (\zert_{ij}+\xi) - \ell (\zert_{ij}-\xi) = 2
d\ell|_{\zert_{ij}}\cdot \xi +
\cO(N^{-3}) 
$$

$$
D^v_{ij} = \ell (\zert_{ij}+\zeta) - \ell (\zert_{ij}-\zeta) = 2
d\ell|_{\zert_{ij}}\cdot \zeta +
\cO(N^{-3}) 
$$
Hence
$$
\omega(D^u_{ij}, D^v_{ij})= 4\omega (d \ell_{\zert_{ij}} \cdot \xi, 
d\ell_{\zert_{ij}} \cdot \zeta)+\cO(N^{-4}).
$$
Since $\ell$ is isotropic, the first term of the RHS vanishes, hence
$$
\omega(D^u_{ij}, D^v_{ij})=\cO(N^{-4})
$$
so that
we have 
\begin{equation}\label{eq:estiC0}
	\|\eta_N\|_{\cC^0}=\cO(N^{-2}).
\end{equation}

We want to prove a similar estimate of the first order 
finite differences
$$
\frac{\del \eta_N}{\del\vec u} \quad \mbox{ and } \quad \frac
{\del\eta_N}{\del\vec v}.
$$
By definition of finite differences (cf. Formulas~\eqref{eq:fd})
\begin{equation*}
\frac{\del\eta_N}{\del\vec u} (\face_{ij}) = \frac {N^3}{\sqrt
2^3}\big (
\omega(D^u_{i+1,j+1}, D^v_{i+1,j+1})
- \omega(D^u_{i,j}, D^v_{i,j})\big )
\end{equation*}
which can be expressed as
\begin{align}
 \frac {\sqrt
2^3}{N^3}
	\frac{\del\eta_N}{\del\vec u} (\face_{ij}) =
	\omega(D^u_{i+1,j+1}-D^u_{ij},& D^v_{i+1,j+1}) \nonumber\\
	& + \omega(D^u_{i,j}, D^v_{i+1,j+1}-D^v_{i,j})
	\label{eq:taylor0}
\end{align}

The Taylor formula applied at the intersection point
$\vert=\vert_{i+1,j+1}$ between the facets $\face_{ij}$ and
$\face_{i+1,j+1}$,
gives
\begin{align}
	D^u_{ij} &=\ell (\vert)- \ell(\vert - 2\xi)  \nonumber \\
	&= d\ell_\vert \cdot 2\xi   
	-\frac 12 d^2\ell_\vert\cdot 2\xi 
	+\frac 16 d^3\ell_\vert\cdot 2\xi +\cO(N^{-4}) \nonumber \\
	&= \sqrt 2 N^ {-1} \frac{\del \ell}{\del u}(\vert)-
	N^{-2}\frac{\del^2\ell}{\del u^2}(\vert)+ 
	\frac {\sqrt 2}3 N^{-3} \frac{\del^3 \ell}{\del u^3}(\vert)
	+ \cO(N^{-4}) \label{eq:taylor1}
\end{align}
and
\begin{align}
	D^u_{i+1,j+1} &= \ell(\vert + 2\xi) -\ell(\vert) \nonumber \\
	&= d\ell_\vert \cdot 2u 
	+\frac 12 d^2\ell_\vert \cdot 2\xi
	+\frac 16 d^3\ell_\vert \cdot 2\xi +\cO(N^{-4}) \nonumber \\
	&= \sqrt 2 N^ {-1} \frac{\del \ell}{\del u}(\vert)+
	N^{-2}\frac{\del^2\ell}{\del u^2}(\vert)+ 
	\frac {\sqrt 2}3 N^{-3} \frac{\del^3 \ell}{\del u^3}(\vert)
	+ \cO(N^{-4}) \label{eq:taylor2}.
\end{align}
Hence by Equations~\eqref{eq:taylor1} and \eqref{eq:taylor2}, we have
\begin{equation}
	\label{eq:taylor3}
	D^u_{i+1,j+1}-D^u_{ij} =
	2 N^ {-2}\frac{\del^2\ell}{\del u^2}(\vert)+\cO(N^{-4})
\end{equation}
Using the notations
$$
k_1=\frac {e_1}N =N^{-1}\frac\del{\del x} \quad \mbox{ and }\quad k_2=\frac{e_2}N=
N^{-1}\frac \del{\del y},
$$
a similar computation shows that
\begin{multline*}
	D^v_{ij} = \ell(\vert - k_1)-\ell(\vert-k_2)
	 = d\ell_\vert\cdot (k_2-k_1)\\ + \frac 12 d^2\ell_\vert\cdot k_1 
	-   \frac 12 d^2\ell_\vert\cdot k_2
	 - 
	\frac 16 d^3\ell_\vert \cdot k_1 +
	\frac 16 d^3\ell_\vert\cdot k_2 + \cO(N^{-4}) 
\end{multline*}
and using the fact that $k_2+k_1=2\xi$, $k_2-k_1=2\zeta$, we have
$$
	D^v_{ij} = d\ell_\vert \cdot 2\zeta -\frac 12 d^2\ell_\vert (2\zeta,2\xi)
	-\frac 16 d^3\ell_\vert\cdot k_1 
	+\frac 16 d^3\ell_\vert \cdot k_2 +\cO(N^{-4}) 
	$$
	so that
	\begin{multline}
		D^v_{ij} = \sqrt 2 N ^{-1}\frac {\del \ell}{\del u}(\vert)
	-N^{-2}\frac{\del^2\ell}{\del u\del v}(\vert)
\\
		+\frac {N^{-3}}6 \left (\frac {\del^3\ell}{\del y^3}(\vert)
	-\frac {\del^3\ell}{\del x^3}(\vert)\right )+\cO(N^{-4})
	\label{eq:taylor4}
\end{multline}
Now
\begin{multline*}
	D^v_{i+1,j+1} = \ell(\vert + k_2)-\ell(\vert+k_1) 
	=  d\ell_\vert\cdot (k_2-k_1)+\frac 12 d^2\ell_\vert \cdot k_2 
	-\frac 12 d^2\ell_\vert\cdot k_1 \\+ \frac 16 d^3\ell_\vert \cdot k_2 
	-\frac 16 d^3\ell_\vert\cdot k_1 + \cO(N^{-4}) 
\end{multline*}
and we have
\begin{multline}
	D^v_{i+1,j+1}	= \sqrt 2N^{-1}\frac{\del \ell}{\del u}(\vert)
	+N^{-2}\frac{\del^2\ell}{\del u\del v}(\vert)
\\	+\frac {N^{-3}}6 \left (\frac {\del^3\ell}{\del y^3}(\vert)
	-\frac {\del^3\ell}{\del x^3}(\vert)\right )+\cO(N^{-4})
	\label{eq:taylor5}
\end{multline}
Then, by Equations~\eqref{eq:taylor4} and \eqref{eq:taylor5}
\begin{equation}
	\label{eq:taylor6}
	D^v_{i+1,j+1}-D^v_{ij} =
	2N ^{-2}\frac {\del^2\ell}{\del u\del v}(\vert)+\cO(N^{-4}).
\end{equation}
In conclusion, by Equations~\eqref{eq:taylor0},\eqref{eq:taylor1},
\eqref{eq:taylor3}, \eqref{eq:taylor5} and 
\eqref{eq:taylor6}
\begin{multline}\label{eq:taylor7}
	{ N^{-3}} \frac{\del\eta_N}{\del \vec u}(\face_{ij})=
N^{-3}\omega \left (\frac{\del^2\ell}{\del u^2}(\vert)
,\frac{ \del \ell}{\del v}(\vert) \right) +
\\
	N^{-3}\omega \left (
\frac{ \del \ell}{\del u}(\vert) ,
\frac{\del^2\ell}{\del u\del v}(\vert)
\right ) +\cO(N^{-5}).
\end{multline}
The lower order term of the above expansion vanishes. Indeed,
by isotropy of $\ell$, we have the equation
$$
0=\omega \left (\frac{\del \ell}{\del u},\frac{\del \ell}{\del v}\right )
$$
and differentiating in the $u$-direction gives
$$
0=\omega\left (\frac{\del^2 \ell}{\del u^2},\frac{\del \ell}{\del v}\right )+
\omega\left (\frac{\del \ell}{\del u},
\frac{\del^2 \ell}{\del u\del v}\right ).
$$
Therefore 
\begin{equation}\label{eq:estderu}
\frac{\del\eta_N}{\del \vec u}(\face_{ij})=\cO(N^{-2}).
\end{equation}
It is easy to see that the estimate $\cO(N ^{-2})$ are uniform in
$\face_{ij}$, and that
the constant involved depend only on the derivatives of $\ell$, which are
bounded since $\Sigma$ is closed.
By symmetry, we have similar estimates
in the $v$-direction 
\begin{equation}
	\label{eq:estider}
\frac{\del\eta_N}{\del \vec v}(\face_{ij})=\cO(N^{-2}),
\end{equation}
	so that by Equations~\eqref{eq:estiC0}, \eqref{eq:estderu} and~\eqref{eq:estider} 
$$
\|\eta_N\|_{\cC^1_w} = \cO(N^{-2})
$$
which proves the proposition.
\end{proof}
\begin{cor}
	\label{cor:quadest}
	The isotropic quadrangular meshes $\rho_N$ of Theorem~\ref{theo:jrt}
	satisfy 
	$$
	\|\rho_N-\tau_N\|_{\cC^0}=\cO(N^{-2}).
	$$
\end{cor}
\begin{proof}
	This is an immediate consequence of Theorem~\ref{theo:jrt} and
	Proposition~\ref{prop:betterest}.
\end{proof}

\subsection{Triangular mesh estimates}
The goal of this section is to expand the estimate of
Corollary~\ref{cor:quadest} into an estimate for triangular
meshes as follows:
\begin{prop}
	\label{prop:triest}
	The optimal isotropic triangular mesh $\rho'_N$ and the triangular
	samples $\tau'_N$ of $\ell$ satisfy the estimate
$$
	\|\rho'_N-\tau'_N\|_{\cC^0}=\cO(N^{-2}).
$$
\end{prop}
\begin{lemma}
	\label{lemma:triest1}
Let $\hat \rho_N$ and $\tau'_N$ be the  triangular meshes
	constructed from $\rho_N$ and $\tau_N$. Then
	$$
	\|\hat \rho_N - \tau'_N\|_{\cC^0}=\cO(N^{-2}).
	$$
\end{lemma}
\begin{proof}[Proof of Lemma~\ref{lemma:triest1}]
	Let $\face_{ij}$ be a facet of $\cQ_N$ and 
	$\zert_{ij}$ its barycenter.
	By definition $\hat\rho_N(\zert_{ij})$ is the barycenter in
	$\RR^{2n}$ of the points $\rho_N(\vert)$ where $\vert$ is  a
	vertex of $\face_{ij}$.
Using the Taylor formula we have
$$
	\sum_{\vert\sim\zert_{ij}} \ell(\vert) = 4\ell(\zert_{ij}) 
	+ \cO(N^{-2}),
$$
	where $\zert_{ij}\sim \vert$ means that $\vert$ is a vertex of
	$\face_{ij}$.
	Since $\ell(\vert)=\tau_N(\vert)=\rho_N(\vert)+\cO(N^{-2})$ by
	Corollary~\ref{cor:quadest}, we deduce that
	$$
	\hat\rho_N(\zert_{ij})=\ell(\zert_{ij})+\cO(N ^{-2})
	$$
	and since $\tau'_N( \zert_{ij})=\ell(\zert_{ij})$, by definition,
	this proves the proposition.
\end{proof}

\begin{prop}
	\label{lemma:triest2}
$$
	\|\rho'_N-\hat \rho_N\|_{\cC^0}=\cO(N ^{-2})
$$
\end{prop}
\begin{proof}[Proof of Proposition~\ref{lemma:triest2}]
The optimal isotropic triangular mesh $\rho'_N$ is obtained by choosing the
	closest point to $\hat \rho_N(\zert_{ij})$, such that the pyramid
	constructed from the quadrilateral associated to the facet $\face_{ij}$ of the quadrangular mesh
	$\rho_N$ is isotropic. 
	We merely need to prove that
	$$
	\hat\rho_N(\zert_{ij}) - \rho'_N(\zert_{ij})=\cO(N^{-2}),
	$$
	since $\rho_N'$ and $\hat\rho_N$ agree along vertices of $\cQ_N$.

	Recall that $\vert_{ij}$ is a vertex of $\face_{ij}$. We consider
	the \emph{parallelogram}  $(B_0B_1B_2B_3)$ of $\RR^{2n}$ given by
	$B_0=\ell(\vert_{ij} ) = \tau'_N(\vert_{ij})$ and
	$$
	\overrightarrow{B_0B_1} = \frac{\del\ell}{\del x}(\vert_{ij}),
	\quad \overrightarrow{B_0B_3} =  \frac{\del\ell}{\del
	y}(\vert_{ij}).
	$$
	Thus, $(B_0B_1B_2B_3)$ is an isotropic parallelogram tangent to the
	map $\ell$ at $\ell(\vert_{ij})$.
	\begin{lemma}\label{lemma:as}
		All the parallelograms $(B_0B_1B_2B_3)$ obtained by the above
	construction	are uniformly close to be squares,
		in the following sense:
there exists  constants $c_1, c_2>0$, independent of $N$ and the choice of
		facet
	$\face_{ij}$, such that the side lengths of the parallelogram
	$(B_0B_1B_2B_3)$, belongs to the interval $(c_1,c_2)$ and
		$$g(\overrightarrow{B_iB_{i+1}},
		\overrightarrow{B_{i+1}B_{i+2}}) = \cO(N^{-1}),$$
		where the index $i$ is understood modulo $4$.
	\end{lemma}
	\begin{proof}[Proof of Lemma~\ref{lemma:as}]
		The fact that $\ell$ is a smooth immersion defined on a compact
		surface implies the existence of the positive constants
		$c_1,c_2>0$.
		The covering maps  $p\circ r_N:\RR^2\to \Sigma$ are
		\emph{almost} conformal in the sence of
		Equation~\eqref{eq:ac}, which implies the almost orthogonality of
		consecutive sides up to an error term of order~$\cO(N^{-1})$.
	\end{proof}
	We return to the proof of Proposition~\ref{lemma:triest2}.
	We  consider the quadrilateral $(C_0C_1C_2C_3)$ defined by
	the facet $\face_{ij}$ and $\tau_N$ rescaled by a factor~$N$:
	we put $C_0=B_0 =\tau_N(\vert_{ij})$ and the other points $C_i$ 
	are given by the vertices of
	the mesh $\tau_N$ around the facet $\face_{ij}$, rescaled by a
	homothety of center $C_0$ and
	scaling factor~$N$.

	Similarly, we construct a quadrilateral $(A_0A_1A_2A_3)$ associated
	to the isotropic quadrangular mesh $\rho_N$ and the facet $\face_{ij}$,
	rescaled by the same homothety.

	It follows from Corollary~\ref{cor:quadest}
	that the quadrilaterals $(A_0A_1A_2A_3)$ and $(C_0C_1C_2C_3)$ agree
	up to a perturbation of order $\cO(N^{-1})$ (the loss of one order is due to
	the rescaling).
	By the Taylor formula $(B_0B_1B_2B_3)$ agrees with $(C_0C_1C_2C_3)$
	up to a perturbation of order $\cO(N^{-1})$. In conclusion
	$(A_0A_1A_2A_3)$ and $(B_0B_1B_2B_3)$ agree up to a perturbation of order
	$\cO(N^{-1})$ and we deduce from Lemma~\ref{lemma:as} that the quadrilateral
		$(A_0A_1A_2A_3)$ is almost a
	square in the same sense as $(B_0B_1B_2B_3)$. 

	The optimal apex 
	of~$(B_0B_1B_2B_3)$, in the sense of Definition~\ref{dfn:optimal}, agrees with its barycenter,
	since this is  a parallelogram contained in an isotropic plane. 
	We expect that a small isotropic deformation of this
	quadrilateral like
	$(A_0A_1A_2A_3)$ ought to have an optimal apex  very close to its
	barycenter as well. 
This is indeed to case: if the isotropic quadrilateral~$(A_0A_1A_2A_3)$ is contained in a plane, then this plane must
	be isotropic and the optimal apex agrees with the barycenter of
	$(A_0A_1A_2A_3)$, by definition. 
	If~$(A_0A_1A_2A_3)$ is not contained in a plane,  
 an explicit  apex that completes the
	quadrilateral $(A_0A_1A_2A_3)$ into an isotropic pyramid can be found
	by solving the linear system of equations. This work is carried out
	in details in
	\cite[\S 7.1]{JRT}, where a 
	particular solution $P$ is given explicitly by \cite[Formula (7.16)]{JRT}.
The fact that $(A_0A_1A_2A_3)$
is almost a square in the sense of Lemma~\ref{lemma:as} implies that 
$$\alpha_0=\frac 12 +\cO(N^{-1}), \quad \alpha_1 = \frac 12 +\cO(N^{-1}),\quad \beta_0 =
\cO(N^{-1}) \quad \mbox { and } \quad \beta_1=\cO(N^{-1}).
$$
where the notations of \cite[\S 7.1.5]{JRT} have been used in the above
identities.

Then by \cite[Formula (7.15)]{JRT}, we have
$\xi(V)=(|\beta_0|+|\beta_1|)\cO(N^{-1})$, and we deduce from
\cite[Formula (7.16)]{JRT} that the particular solution $P$  satisfies
$$
\overrightarrow{GP} = \cO(N^{-1})
$$
where $G$ is the barycenter of $(A_0A_1A_2A_3)$.
In conclusion, the barycenter of the isotropic quadrilateral $(A_0A_1A_2A_3)$
	agrees with the optimal apex, up to  a perturbation of order
	$\cO(N^{-1})$. After scaling back to the original picture by a factor
	$N^{-1}$, this proves the proposition.
\end{proof}

\begin{proof}[Proof of Proposition~\ref{prop:triest}]
The proposition is an immediate consequence of Lemma~\ref{lemma:triest1}
and Proposition~\ref{lemma:triest2}.
\end{proof}

\subsection{Proof of the main theorem}
\begin{prop}
	\label{prop:plest}
	The optimal piecewise linear isotropic maps
	$\ell_N:\Sigma\to\RR^{2n}$ (cf. Definition~\ref{dfn:optpl}) satisfy the estimate
	$$
	\|\ell-\ell_N\|_{C^0} = \cO(N^{-2}), \quad
	\|\ell-\ell_N\|_{C^1} = \cO(N^{-1}).
	$$
\end{prop}
\begin{proof}
	The piecewise linear isotropic map~$\ell_N$ is defined as an affine map on each facet of the
	triangulation~$\scrT_N$, which agrees with the isotropic triangular
	mesh~$\rho'_N$ at the vertices (cf. Definition~\ref{dfn:optpl}).

	Another piecewise linear map $f_N:\Sigma \to \RR^{2n}$ can be
	associated to the mesh~$\tau'_N$ exactly in the same way. In other
	words, $f_N$ is an affine map on each facet of the triangulation~$\scrT_N$, which agrees with
	the map $\ell$ at the vertices of the triangulation. This type of
	approximation is classical and by the Taylor formula applied on each
	facet of the triangulation, we have the
	following lemma:
	\begin{lemma} 
		\label{lemma:approx1}
		The sequence of piecewise linear maps $f_N$
		approximates  the
		smooth map $\ell$, in the sense that 
		$$\|f_N-\ell\|_{C^0}=\cO(N^{-2}).$$
Furthermore, $f_N$ piecewise smooth, with $\scrT_N$ as an adapted
		triangulation  and
		we have 
		$$\|df_N-d\ell \|_{C^0}=\cO(N^{-1}),$$ where 
		the $C^0$-norm of differentials of piecewise smooth maps is defined in the sense of
		Formula~\eqref{eq:c0ps}.
	\end{lemma}
The restriction of~$f_N$ and~$\ell_N$ along a facet of the
	triangulation is an  affine map. The values of the maps at the vertices are
	given respectively by~$\tau'_N$ and~$\rho'_N$. By
	Proposition~\ref{prop:triest}, these
	control values agree up to an error of order~$\cO(N^{-2})$. Since
	the  facets have size~$N^{-1}$, 
	the estimate extends globally on the entire facet
	and we obtain the control
	$$
	\|\ell_N-f_N\|_{C^0}= \cO(N^{-2}).
	$$
	Together with~Lemma~\ref{lemma:approx1}, we deduce the estimate
	\begin{equation}
		\label{eq:fs}
			\|\ell_N-\ell\|_{C^0}= \cO(N^{-2}),
	\end{equation}
	which proves the first statement of the proposition.
	
	Rescaling the source and target spaces by a factor $N$ for the restriction
	of maps $u_N=f_N-\ell_N$  on a facet of the triangulation, shows a
	sequence of affine maps, whose values 
	are of order $\cO(N^{-1})$ at the vertices. We deduce that the 
	 linear part of the maps
	$u_N$ along the facet must be of the same order. It follows that
	$$\|df_N-d\ell_N\|_{C^0}=\|du_N\|_{C^0}=\cO(N^{-1})$$
	where the $C^0$-norm of differential of  piecewise smooth
	maps is taken in the sense of Formula~\eqref{eq:c0ps}.
	Together with~Lemma~\ref{lemma:approx1}, we deduce the estimate
	\begin{equation}
		\label{eq:ss}
	\|d\ell_N-d\ell\|_{C^0}= \cO(N^{-1}).
	\end{equation}
	Formula~\eqref{eq:fs} and Formula~\eqref{eq:ss} give the control
	$$
	\|\ell_N-\ell\|_{C^1}=\cO(N^{-1})
	$$
	which proves the second statement of the proposition.
\end{proof}
\begin{cor}
	\label{cor:plest}
	For every sufficiently large $N$, the piecewise linear isotropic
	map $\ell_N:\Sigma\to\RR^{2n}$ is a topological immersion (resp. embedding)
	if $\ell$ is an smooth immersion (resp. embedding).
\end{cor}
\begin{proof}
The sequence of optimal isotropic piecewise linear maps $\ell_N$ satisfies
	$$
	\lim_{N\to\infty}\|\ell-\ell_N\|_{C^1}= 0
	$$
	by Proposition~\ref{prop:plest}. 
	If $\ell$ is a smooth immersion (resp. embedding), this forces
	$\ell_N$ to be a topological immersion (resp. embedding), for every
	sufficiently large $N$, by
	Corollary~\ref{cor:approx}, which proves the theorem.
\end{proof}
\begin{proof}[Proof of Theorem~\ref{theo:main}]
	The theorem is an immediate consequence of
	Proposition~\ref{prop:plest} and Corollary~\ref{cor:plest}.
\end{proof}
\vspace{10pt}
\bibliographystyle{abbrv}
\bibliography{polylag}

\end{document}